\documentclass{article}

\usepackage{PRIMEarxiv}

\usepackage[utf8]{inputenc} % allow utf-8 input
\usepackage[T1]{fontenc}    % use 8-bit T1 fonts
\usepackage{hyperref}       % hyperlinks
\usepackage{url}            % simple URL typesetting
\usepackage{booktabs}       % professional-quality tables
\usepackage{amsfonts}       % blackboard math symbols
\usepackage{nicefrac}       % compact symbols for 1/2, etc.
\usepackage{microtype}      % microtypography
\usepackage{lipsum}
\usepackage{fancyhdr}       % header
\usepackage{graphicx}       % graphics
\graphicspath{{media/}}     % organize your images and other figures under media/ folder
\usepackage{amsmath}
\usepackage{amsthm}
\usepackage{float}
\usepackage{tikz}  

\theoremstyle{plain} %% This is the default
\newtheorem{theorem}{Theorem}[section]
\theoremstyle{definition}
\newtheorem{corollary}[theorem]{Corollary}
\newtheorem{lemma}[theorem]{Lemma}
\newtheorem{proposition}{Proposition}[section]
\newtheorem{definition}{Definition}[section]
\newtheorem{remark}{Remark}[section]
\newtheorem{example}{Example}

\title{$A_\alpha$-energy of graphs formed by some unary operations
}

\author{
  Najiya V K, Chithra A V \\
 Department of Mathematics \\
 National Institute of Technology, Calicut\\
 Kerala, India-673601\\
  \texttt{najiya\_p190046ma@nitc.ac.in, chithra@nitc.ac.in} \\
}

\begin{document}
\maketitle

\begin{abstract}
 Let $G $ be a graph on $p$ vertices with adjacency matrix $A(G)$ and degree matrix $D(G)$. For
each $\alpha \in [0, 1]$, the $A_\alpha$-matrix is defined as $A_\alpha (G) = \alpha D(G) + (1 - \alpha)A(G)$. In this paper, we compute the $A_\alpha$-characteristic polynomial, $A_\alpha$-spectra and $A_\alpha$-energy of some non-regular graphs obtained from unary operations on graphs like middle graph, central graph, m-splitting, and closed splitting graph. Also, we determine the $A_\alpha$-energy of regular graphs like m-shadow, closed shadow, extended bipartite double graph, iterated line graph and m-duplicate graph. Furthermore, we identified some graphs that are $A_\alpha$-equieneregetic and $A_\alpha$-borderenergetic.
\end{abstract}

% keywords can be removed

\section{Introduction}
Let $G =(V (G), E(G))$ be a simple connected undirected graph with the vertex set $V (G)=\{v_1,v_2,\dots,v_p\}$ and the edge set $E(G)=\{e_1,e_2,\dots,e_q\}$. The adjacency matrix $A(G)$ of $G$ is a $p \times p$ symmetric matrix defined as $$[A(G)]_{i,j}=\begin{cases}
    1 & \text{if } v_i \text{ and } v_j \text{ are adjacent}\\
    0 & \text{otherwise}.
\end{cases}$$ The degree matrix $D(G)$ is the $p \times p$ diagonal matrix, such that $$[D(G)]_{i,j}=\begin{cases}
    deg(v_i) & \text{if } i=j\\
    0 & \text{otherwise},
\end{cases}$$ where $deg(v)$ is the degree of vertex $v$ in $G$. The incidence matrix $R(G)$ is the $(0, 1)$-matrix, whose rows and columns are indexed by the vertex and edge sets of $G$, such that $$[R(G)]_{i,j}=\begin{cases}
    1 & \text{if } v_i \text{ and } e_j \text{ are incident}\\
    0 & \text{otherwise}.
\end{cases}$$ $R(G)R(G)^T = A(G) + D(G) $ and $R(G)^TR(G) = B(G) + 2I_q $, where $B(G)$ is the adjacency matrix of line graph of $G$. 

The $A_\alpha$-matrix\cite{nikiforov2017merging}, $A_\alpha(G) = \alpha D(G) + (1 -\alpha)A(G)$, for $\alpha \in  [0, 1]$ is a convex combination of the adjacency and degree matrix of a graph $G$. It is clear that $A_0 (G)=A(G), \quad A_{\frac{1}{2}}(G)=\frac{1}{2}Q(G)$ and $A_1 (G)=D(G)$. Also, for $\alpha,\beta\in[0,1]$, $A_\alpha (G)-A_\beta(G) = (\alpha-\beta)L(G) =(\alpha-\beta)(D(G)- A(G))$, where $L(G)$ is the Laplacian matrix of $G$. $A_\alpha$-matrix helps to study the spectral properties of uncountable many convex combinations of $D(G)$ and $A(G)$. For a $p\times p$ matrix $M$, let $det(M)$ and $M^T$ denote the determinant and the transpose of $ M$, respectively. We denote the characteristic polynomial of $M$ as $\phi (M,\lambda) = det(\lambda I_p - M )$, where $I_p$ is the identity matrix of order $p$. The roots of the $M$-characteristic polynomial of $G$ are the $M$-eigenvalues of $G$. Let $\lambda_i (A(G)) $ and $\lambda_i (A_\alpha (G)), i=1,2,\dots,p,$ be the adjacency and $A_\alpha $-eigenvalues of $G$, respectively. The collection of all eigenvalues of $A(G)$ and $A_\alpha (G)$, including multiplicities, is called the $A$-spectrum and the $A_\alpha $-spectrum of $G$, respectively. The $A_\alpha$-spectrum of $G$ with $k$ distinct eigenvalues can be written as
$$\sigma_{A_\alpha}\left(G\right)=\begin{pmatrix}
\lambda_1(A_\alpha (G)) & \lambda_2(A_\alpha (G)) & \cdots &\lambda_k(A_\alpha (G))\\
m_1 & m_2& \cdots & m_k
\end{pmatrix},$$
where $m_i$ is the algebraic multiplicity of $\lambda_i(A_\alpha (G))$, for $1 \leq i \leq k$. 

The sum of absolute values of the adjacency eigenvalues, $\displaystyle\sum_{i=1}^p|\lambda_i(A(G)|$, of a graph gives the adjacency energy, $\varepsilon\left(G\right)$, of the graph. The $A_\alpha$-energy\cite{pirzada2021alpha} of a graph $G$, for $\alpha\in\left[0,1\right)$, is defined as $\varepsilon_\alpha\left(G\right)=\displaystyle\sum_{i=1}^p \left|\lambda_i\left(A_\alpha\left(G\right)\right)-\frac{2\alpha q}{p}\right|.$ If the graph $G$ is regular, then $\varepsilon_\alpha\left(G\right)=\left(1-\alpha\right)\varepsilon\left(G\right)$.

If two graphs have the same $A_\alpha$-energy for some value of $\alpha\in[0,1)$, they are said to be $A_\alpha$-equienergetic for that value of $\alpha$. In \cite{najiya2023study} authors introduced the concept of $A_\alpha$-borderenergetic and $A_\alpha$-hyperenergetic graphs. A graph $G$ on $n$ vertices is $A_\alpha$-borderenergetic if $\varepsilon_\alpha(G) = \varepsilon_\alpha(K_n )$, for some $\alpha\in[0,1)$. Borderenergetic graphs are not $A_\alpha$-borderenergetic, but regular borderenergetic graphs are $A_\alpha$-borderenergetic for every value of $\alpha$. The graphs whose $A_\alpha$-energy exceeds the $A_\alpha$-energy of the complete graph on the same vertices are called $A_\alpha$-hyperenergetic. That is, a graph $G$ is $A_\alpha$-hyperenergetic if $\varepsilon_\alpha(G) \geq \varepsilon_\alpha(K_n )$, for some $\alpha\in[0,1)$.

In this paper, $K_p$ and $K_{p,q}$ denote the complete graph and the complete bipartite graph, respectively. $0_{ p\times q}$ and $J_{p\times q}$ denote the matrices of order $p\times q$ consisting of all 0 and all 1, respectively.

This paper is structured in the following manner; Section \ref{prelim} presents various definitions and results essential to proving the results.
In Section \ref{main}, we present the main results obtained for the $A_\alpha$-characteristic polynomial and spectrum of some unary operations on graphs. 

\section{Preliminaries}\label{prelim}
In this section, we state some definitions and lemmas that will be used to prove our main results.
\begin{definition}\cite{hamada1976traversability}
    Let $G = (V (G), E(G))$ be a simple graph. The middle graph $M(G)$ of a graph $G$ is the graph whose vertex set is $V (G) \cup E(G)$ and two vertices $u, v$ in the vertex set of $M(G)$ are adjacent in $M(G)$ in case one the following holds:
    \begin{enumerate}
        \item $u, v$ are in $E(G)$ and $u, v$ are adjacent in $G$.
        \item $u$ is in $V (G)$, $v$ is in $E(G)$, and $u, v$ are incident in $G$.
    \end{enumerate}
\end{definition}

\begin{figure}[H]%\label{alcoreg}
\begin{center}
\begin{tikzpicture}[scale=1,auto=center] 
 \tikzset{dark/.style={circle,fill=black}}
 \tikzset{red/.style={circle,fill=red}}
 \tikzset{white/.style={circle,draw=white}}% here, node/.style is the style pre-defined, that will be the default layout of all the nodes. You can also create different forms for different nodes.  
    \node [dark] (a1) at (0,0){} ;  
  \node [dark] (a2) at (0,1)  {}; 
  \node [dark] (a3) at (1,1)  {};  
  \node [dark] (a4) at (1,0) {}; 
  
  \node [dark] (a5) at (5,0)  {};  
  \node [dark] (a6) at (6,0)  {};  
  \node [dark] (a7) at (6,1)  {};  
  \node [dark] (a8) at (5,1){};
  
  \node [red] (a9) at (5.5,-1) {};
  \node [red] (a10) at (7,0.5) {};
   \node [red] (a11) at (5.5,2){} ;  
  \node [red] (a12) at (4,0.5)  {};

  % \node  at (0.5,-2) {(a) $C_4$};
  %  \node  at (5.5,-2) {(b) $M(C_4)$};
  
  \draw (a1) -- (a2);
  \draw (a2) -- (a3);  
  \draw (a3) -- (a4);  
  \draw (a4) -- (a1);  
  
  \draw (a5) -- (a9);  
  \draw (a5) -- (a12);  
  \draw (a6) -- (a9);
  \draw (a6) -- (a10);  
  \draw (a7) -- (a10);  
  \draw (a7) -- (a11);  
  \draw (a8) -- (a11);  
  \draw (a8) -- (a12);  
  \draw[red, very thick] (a9) -- (a10);   
  \draw[red, very thick] (a10) -- (a11);  
  \draw[red, very thick] (a11) -- (a12); 
  \draw[red, very thick] (a12) -- (a9);
  
\end{tikzpicture}  
\end{center}
\caption{$C_4$ and $M(C_4)$} \label{middle}
\end{figure}
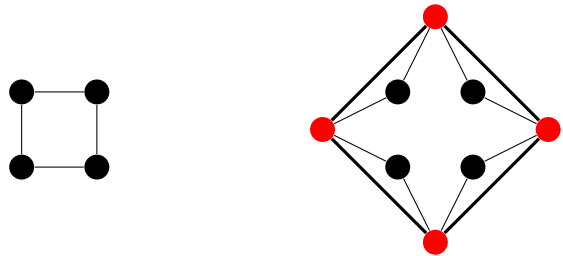

\begin{definition}\cite{vivin2008harmonius}
    Let $G$ be a simple graph with $p$ vertices and $q$ edges. The central graph of $G$, $C(G)$ is obtained by subdividing each edge of $G$ exactly once and joining all the non-adjacent vertices in $G$.
\end{definition}

\begin{figure}[H]%\label{alcoreg}
\begin{center}
\begin{tikzpicture}[scale=1,auto=center] 
 \tikzset{dark/.style={circle,fill=black}}
 \tikzset{red/.style={circle,fill=red}}
 \tikzset{white/.style={circle,draw=white}}% here, node/.style is the style pre-defined, that will be the default layout of all the nodes. You can also create different forms for different nodes.  

  \node [dark] (a5) at (5,0)  {};  
  \node [dark] (a6) at (6,0)  {};  
  \node [dark] (a7) at (6,1)  {};  
  \node [dark] (a8) at (5,1){};
  
  \node [red] (a9) at (5.5,-1) {};
  \node [red] (a10) at (7,0.5) {};
   \node [red] (a11) at (5.5,2){} ;  
  \node [red] (a12) at (4,0.5)  {};

   % \node  at (5.5,-2) {(c) $C(C_4)$};

  \draw[red, very thick] (a5) -- (a9);  
  \draw[red, very thick] (a5) -- (a12);  
  \draw[red, very thick] (a6) -- (a9);
  \draw[red, very thick] (a6) -- (a10);  
  \draw[red, very thick] (a7) -- (a10);  
  \draw[red, very thick] (a7) -- (a11);  
  \draw[red, very thick] (a8) -- (a11);  
  \draw[red, very thick] (a8) -- (a12);  
  \draw (a5) -- (a7);   
  \draw (a6) -- (a8);

\end{tikzpicture}  
\end{center}
\caption{$C(C_4)$} \label{central}
\end{figure}

\begin{definition}\cite{vaidya2017energy}
    The $m$-splitting graph $Spl_m (G )$ of a graph $G$ is obtained by adding $m$ new vertices, say $v_1, v_2 ,\dots, v_m$ to each vertex $v$ of $G$, such that $v_i$ is adjacent to each vertex that is adjacent to $v$ in $G$.
\end{definition}

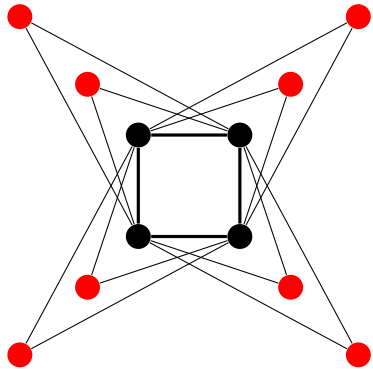
\begin{figure}[H]%\label{alcoreg}
\begin{center}
\begin{tikzpicture}[scale=0.9,auto=center] 
 \tikzset{dark/.style={circle,fill=black}}
 \tikzset{red/.style={circle,fill=red}}
 \tikzset{white/.style={circle,draw=white}}% here, node/.style is the style pre-defined, that will be the default layout of all the nodes. You can also create different forms for different nodes.  

  \node [dark] (a5) at (-0.25,-0.25)  {};  
  \node [dark] (a6) at (1.25,-0.25)  {};  
  \node [dark] (a7) at (1.25,1.25)  {};  
  \node [dark] (a8) at (-0.25,1.25){};
  
  \node [red] (a9) at (-1,-1) {};
  \node [red] (a10) at (-2,-2) {};
   \node [red] (a11) at (2,-1){} ;  
  \node [red] (a12) at (3,-2)  {}; 
  \node [red] (a13) at (2,2) {};
  \node [red] (a14) at (3,3) {};
   \node [red] (a15) at (-1,2){} ;  
  \node [red] (a16) at (-2,3)  {};

   % \node  at (5.5,-1) {(d) $Spl_2(C_4)$};

  \draw[red, very thick] (a5) -- (a6);  
  \draw[red, very thick] (a6) -- (a7);  
  \draw[red, very thick] (a7) -- (a8);
  \draw[red, very thick] (a8) -- (a5);  
  \draw (a9) -- (a8);  
  \draw (a9) -- (a6);  
  \draw (a10) -- (a8);  
  \draw (a10) -- (a6);  
  \draw (a11) -- (a5);   
  \draw (a11) -- (a7);  
  \draw (a12) -- (a5);  
  \draw (a12) -- (a7);  
  \draw (a13) -- (a6);   
  \draw (a13) -- (a8);  
  \draw (a14) -- (a6);  
  \draw (a14) -- (a8);  
  \draw (a15) -- (a7);   
  \draw (a15) -- (a5); 
  \draw (a16) -- (a7);  
  \draw (a16) -- (a5);

\end{tikzpicture}  
\end{center}
\caption{$Spl_2(C_4)$} \label{msplit}
\end{figure}

\begin{definition}\cite{goyal2019new}
    The closed splitting graph $\Lambda(G)$ of a graph $G$ is the graph whose vertex set is $V(G) \cup V^{\prime}(G)$, where $V^{\prime}(G)$ is the copy of $V(G)$ and the edge set is $E(G) \cup\left\{u u^{\prime}: u \in V(G)\right\} \cup\left\{u v^{\prime}: u v \in E(G)\right\}$.
\end{definition}

\begin{figure}[H]%\label{alcoreg}
\begin{center}
\begin{tikzpicture}[scale=1,auto=center] 
 \tikzset{dark/.style={circle,fill=black}}
 \tikzset{red/.style={circle,fill=red}}
 \tikzset{white/.style={circle,draw=white}}% here, node/.style is the style pre-defined, that will be the default layout of all the nodes. You can also create different forms for different nodes.  

  \node [dark] (a5) at (0,0)  {};  
  \node [dark] (a6) at (1,0)  {};  
  \node [dark] (a7) at (1,1)  {};  
  \node [dark] (a8) at (0,1){};
  
  \node [red] (a9) at (-1,-1) {};
   \node [red] (a10) at (2,-1){} ;  
  \node [red] (a11) at (2,2) {};
   \node [red] (a12) at (-1,2){} ;

   % \node  at (5.5,-2) {(b) $C(C_4)$};

  \draw[red, very thick] (a5) -- (a6);  
  \draw[red, very thick] (a6) -- (a7);  
  \draw[red, very thick] (a7) -- (a8);
  \draw[red, very thick] (a8) -- (a5);  
  % \draw[red, very thick] (a9) -- (a10);  
  % \draw[red, very thick] (a10) -- (a11);  
  % \draw[red, very thick] (a11) -- (a12);  
  % \draw[red, very thick] (a12) -- (a9);  
  \draw (a9) -- (a6);   
  \draw (a9) -- (a8);
   \draw (a10) -- (a5);   
  \draw (a10) -- (a7); 
  \draw (a11) -- (a6);   
  \draw (a11) -- (a8);
   \draw (a12) -- (a7);   
  \draw (a12) -- (a5);
  \draw (a5) -- (a9);
  \draw (a6) -- (a10);  
  \draw (a7) -- (a11);
  \draw (a8) -- (a12);

\end{tikzpicture}  
\end{center}
\caption{$\Lambda(C_4)$} \label{closedsplitting}
\end{figure}
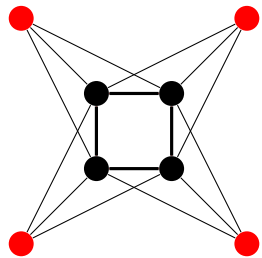

\begin{definition}\cite{vaidya2017energy}
    The $m$-shadow graph $D_m (G )$ of a graph $G$ is obtained by taking $m$ copies of $G$, say $G_1, G_2 , \dots, G_m$, then join each vertex $u$ in $G_i$ to the neighbours of the corresponding vertex $v$ in $G_j$.
\end{definition}

\begin{figure}[H]%\label{alcoreg}
\begin{center}
\begin{tikzpicture}[scale=1,auto=center] 
 \tikzset{dark/.style={circle,fill=black}}
 \tikzset{red/.style={circle,fill=red}}
 \tikzset{white/.style={circle,draw=white}}% here, node/.style is the style pre-defined, that will be the default layout of all the nodes. You can also create different forms for different nodes.  

  \node [dark] (a5) at (0,0)  {};  
  \node [dark] (a6) at (1,0)  {};  
  \node [dark] (a7) at (1,1)  {};  
  \node [dark] (a8) at (0,1){};
  
  \node [red] (a9) at (-1,-1) {};
   \node [red] (a10) at (2,-1){} ;  
  \node [red] (a11) at (2,2) {};
   \node [red] (a12) at (-1,2){} ;

   % \node  at (5.5,-2) {(b) $C(C_4)$};

  \draw[red, very thick] (a5) -- (a6);  
  \draw[red, very thick] (a6) -- (a7);  
  \draw[red, very thick] (a7) -- (a8);
  \draw[red, very thick] (a8) -- (a5);  
  \draw[red, very thick] (a9) -- (a10);  
  \draw[red, very thick] (a10) -- (a11);  
  \draw[red, very thick] (a11) -- (a12);  
  \draw[red, very thick] (a12) -- (a9);  
  \draw (a9) -- (a6);   
  \draw (a9) -- (a8);
   \draw (a10) -- (a5);   
  \draw (a10) -- (a7); 
  \draw (a11) -- (a6);   
  \draw (a11) -- (a8);
   \draw (a12) -- (a7);   
  \draw (a12) -- (a5);

\end{tikzpicture}  
\end{center}
\caption{$D_2(C_4)$} \label{mshadow}
\end{figure}

\begin{definition}\cite{goyal2019new}
    The closed shadowgraph of $G$, denoted by $D_2[G]$, has as the vertex set $V(G) \cup V^{\prime}(G)$, and the edge set $E(G) \cup\left\{u^{\prime} v^{\prime}: u v \in E(G)\right\} \cup\left\{u v^{\prime}: u v \in E(G)\right\} \cup\left\{u u^{\prime}: u \in V(G)\right\}$.
\end{definition}

\begin{figure}[H]%\label{alcoreg}
\begin{center}
\begin{tikzpicture}[scale=1,auto=center] 
 \tikzset{dark/.style={circle,fill=black}}
 \tikzset{red/.style={circle,fill=red}}
 \tikzset{white/.style={circle,draw=white}}% here, node/.style is the style pre-defined, that will be the default layout of all the nodes. You can also create different forms for different nodes.  

  \node [dark] (a5) at (0,0)  {};  
  \node [dark] (a6) at (1,0)  {};  
  \node [dark] (a7) at (1,1)  {};  
  \node [dark] (a8) at (0,1){};
  
  \node [red] (a9) at (-1,-1) {};
   \node [red] (a10) at (2,-1){} ;  
  \node [red] (a11) at (2,2) {};
   \node [red] (a12) at (-1,2){} ;

   % \node  at (5.5,-2) {(b) $C(C_4)$};

  \draw[red, very thick] (a5) -- (a6);  
  \draw[red, very thick] (a6) -- (a7);  
  \draw[red, very thick] (a7) -- (a8);
  \draw[red, very thick] (a8) -- (a5);  
  \draw[red, very thick] (a9) -- (a10);  
  \draw[red, very thick] (a10) -- (a11);  
  \draw[red, very thick] (a11) -- (a12);  
  \draw[red, very thick] (a12) -- (a9);  
  \draw (a9) -- (a6);   
  \draw (a9) -- (a8);
   \draw (a10) -- (a5);   
  \draw (a10) -- (a7); 
  \draw (a11) -- (a6);   
  \draw (a11) -- (a8);
   \draw (a12) -- (a7);   
  \draw (a12) -- (a5);
  \draw (a5) -- (a9);
  \draw (a6) -- (a10);  
  \draw (a7) -- (a11);
  \draw (a8) -- (a12);

\end{tikzpicture}  
\end{center}
\caption{$D_2[C_4]$} \label{closedshadow}
\end{figure}

\begin{definition}\cite{brouwer1989distance}
    Let $G$ be a graph on $p$ vertices. The extended bipartite double graph $Ebd(G)$ of $G$ is the bipartite graph with its partite sets $V_1 = \{v_1 , v_2 , \cdots , v_p \}$ and $V_2 = \{u_1 , u_2 , \cdots , u_p \}$ in which two vertices $v_i$ and $u_j$ are adjacent if $i = j$ or $v_i$ and $u_j$ are adjacent in $G$.
\end{definition}

\begin{figure}[H]%\label{alcoreg}
\begin{center}
\begin{tikzpicture}[scale=1,auto=center] 
 \tikzset{dark/.style={circle,fill=black}}
 \tikzset{red/.style={circle,fill=red}}
 \tikzset{white/.style={circle,draw=white}}% here, node/.style is the style pre-defined, that will be the default layout of all the nodes. You can also create different forms for different nodes.  

  \node [dark] (a5) at (0,0)  {};  
  \node [dark] (a6) at (1,0)  {};  
  \node [dark] (a7) at (1,1)  {};  
  \node [dark] (a8) at (0,1){};
  
  \node [red] (a9) at (-1,-1) {};
   \node [red] (a10) at (2,-1){} ;  
  \node [red] (a11) at (2,2) {};
   \node [red] (a12) at (-1,2){} ;

   % \node  at (5.5,-2) {(b) $C(C_4)$};

  % \draw[red, very thick] (a5) -- (a6);  
  % \draw[red, very thick] (a6) -- (a7);  
  % \draw[red, very thick] (a7) -- (a8);
  % \draw[red, very thick] (a8) -- (a5);  
  % \draw[red, very thick] (a9) -- (a10);  
  % \draw[red, very thick] (a10) -- (a11);  
  % \draw[red, very thick] (a11) -- (a12);  
  % \draw[red, very thick] (a12) -- (a9);  
  \draw (a9) -- (a6);   
  \draw (a9) -- (a8);
   \draw (a10) -- (a5);   
  \draw (a10) -- (a7); 
  \draw (a11) -- (a6);   
  \draw (a11) -- (a8);
   \draw (a12) -- (a7);   
  \draw (a12) -- (a5);
  \draw (a5) -- (a9);
  \draw (a6) -- (a10);  
  \draw (a7) -- (a11);
  \draw (a8) -- (a12);

\end{tikzpicture}  
\end{center}
\caption{$Ebd(C_4)$} \label{ebdg}
\end{figure}

\begin{definition}\cite{harary1969graph}
    The line graph of $G$, $L(G)$, is the graph whose set of vertices corresponds to the set of edges in $G$, where two vertices are adjacent if the corresponding edges in $G$ are adjacent. The $k$-th iterated line graph of $G$ is defined recursively as $L^k(G) = L(L^{k-1}(G))$, $k \geq 2$, where $L(G) = L^1(G)$ and $G = L^0(G)$.
\end{definition}

\begin{figure}[H]%\label{alcoreg}
\begin{center}
\begin{tikzpicture}[scale=1,auto=center] 
 \tikzset{dark/.style={circle,fill=black}}
 \tikzset{red/.style={circle,fill=red}}
 \tikzset{white/.style={circle,draw=white}}% here, node/.style is the style pre-defined, that will be the default layout of all the nodes. You can also create different forms for different nodes.  

  \node [dark] (a1) at (0,4)  {};  
  \node [dark] (a2) at (1,4)  {};  
  \node [dark] (a3) at (2,4)  {};  
  \node [dark] (a4) at (3,4){};
  
  \node [dark] (a5) at (0.5,3) {};
  \node [dark] (a6) at (1.5,3) {};
   \node [dark] (a7) at (2.5,3){} ;  
   
  \node [dark] (a8) at (1,2)  {}; 
  \node [dark] (a9) at (2,2)  {};

  \node [dark] (a10) at (1.5,1)  {};
  
  \node  at (-0.7,4) { $P_4$:};
  \node  at (-1,3) { $L^1(P_4)$:};
  \node  at (-1,2) { $L^2(P_4)$:};
  \node  at (-1,1) { $L^3(P_4)$:};

  \draw (a1) -- (a2);   
  \draw (a2) -- (a3);
   \draw (a3) -- (a4);   
  \draw (a5) -- (a6)--(a7); 
  \draw (a8) -- (a9);

\end{tikzpicture}  
\end{center}
\caption{Iterated line graphs of $P_4$} \label{line}
\end{figure}
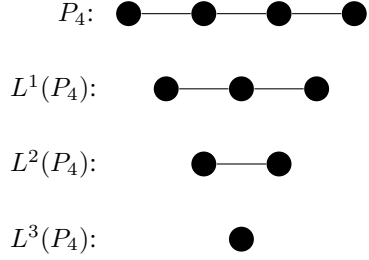

\begin{definition}\cite{sampathkumar1973duplicate}
    Let $G=(V, E)$ be a simple graph. Let $V^{\prime}$ be a set such that $|V|=\left|V^{\prime}\right|, V \cap V^{\prime}=\emptyset$ and $f: V \rightarrow V^{\prime}$ be bijective(for $v\in V$ we write $f(v)=v'$). A duplicate graph of $G$ is $D(G)=\left(V_1, E_1\right)$, where the set of vertices $V_1=V \cup V^{\prime}$ and the set of edges $E_1$ of $D(G)$ is defined as, the edges $uv^{\prime}$ and $u^{\prime} v$ are in $E_1$ if and only if $uv$ is in $E$. In general the $m$-duplicate graph of the graph $G$ is defined as $D^m(G)=D^{m-1}(D(G))$.
\end{definition}

\begin{figure}[H]%\label{alcoreg}
\begin{center}
\begin{tikzpicture}[scale=1,auto=center] 
 \tikzset{dark/.style={circle,fill=black}}
 \tikzset{red/.style={circle,fill=red}}
 \tikzset{white/.style={circle,draw=white}}% here, node/.style is the style pre-defined, that will be the default layout of all the nodes. You can also create different forms for different nodes.  

 \node [dark] (a5) at (-0.25,-0.25)  {};  
  \node [dark] (a6) at (1.25,-0.25)  {};  
  \node [dark] (a7) at (1.25,1.25)  {};  
  \node [dark] (a8) at (-0.25,1.25){};
  
  \node [red] (a9) at (-1,-1) {};
  \node [red] (a10) at (-2,-2) {};
   \node [red] (a11) at (2,-1){} ;  
  \node [red] (a12) at (3,-2)  {}; 
  \node [red] (a13) at (2,2) {};
  \node [red] (a14) at (3,3) {};
   \node [red] (a15) at (-1,2){} ;  
  \node [red] (a16) at (-2,3)  {};
  
   \node [red] (a17) at (-3,-3) {};
  \node [red] (a18) at (4,-3) {};
   \node [red] (a19) at (4,4){} ;  
  \node [red] (a20) at (-3,4)  {};

  \draw (a5) -- (a18);  
  \draw (a5) -- (a20);  
  \draw (a6) -- (a17);
  \draw (a6) -- (a19);  
  \draw (a7) -- (a18);  
  \draw (a7) -- (a20);  
  \draw (a8) -- (a17);  
  \draw (a8) -- (a19);  
  
  \draw (a9) -- (a12);   
  \draw (a9) -- (a16);
  \draw (a10) -- (a11);   
  \draw (a10) -- (a15);
  
  \draw (a11) -- (a14);   
  \draw (a11) -- (a10);
  \draw (a12) -- (a13);   
  \draw (a12) -- (a9);

  \draw (a13) -- (a16);   
  \draw (a14) -- (a15);

\end{tikzpicture}  
\end{center}
\caption{$D^2(C_4)$} \label{duplicate}
\end{figure}

\begin{lemma}\cite{cvetkovic1980spectra}\label{uvwx}
	Let $P,Q,R$ and $S$ be matrices and $$M=\begin{pmatrix}
	P&Q\\
	R&S
	\end{pmatrix}. 
	$$
	
	If $P$ is invertible, then $\det(M)=\det(P)\det(S-RP^{-1}Q)$.
	\medskip
	
	If $S$ is invertible, then $\det(M)=\det(S)\det(P-QS^{-1}R).$ 
	\medskip
	
	If $P$ and $R$ commute, then $\det(M)=\det(PS-QR)$.
\end{lemma}

\section{Main Results}\label{main}
In this section, we derive results related to the computation of the $A_\alpha$-spectrum of some unary operations on graphs. To begin with, we formulate an expression for the $A_\alpha$-characteristic polynomial associated with these operations.

Throughout the section,$G$ is a graph on $p$ vertices and $q$ edges and $A, R,B,$ and $\lambda_i$ represents $A(G), R(G), B(G)$ and $\lambda_i(A(G))$ respectively.

\subsection{Middle Graph}
\begin{proposition}\label{middlep}
    Let $G$ be an $r$-regular graph on $p$ vertices and $q$ edges. Then the $A_\alpha$-characteristic polynomial of middle graph of $G$ is 
    \begin{align*}
    \phi(A_\alpha(M(G))&,\lambda)=\left(\lambda-2\alpha r+2(1-\alpha)\right)^{q-p}\\
    &\prod_{i=1}^p\left(\lambda^2-((1-\alpha)(\lambda_i-2)+r(1+2\alpha))\lambda+r(\alpha^2(r-\lambda_i+1)+\alpha(r+\lambda_i)-1)-(1-\alpha)^2\lambda_i\right).
  \end{align*}
\end{proposition}

\begin{proof}
\begin{align*}
\intertext{The $A_\alpha$ matrix of the middle graph of a regular graph is of the form}
A_\alpha(M(G))&=\begin{bmatrix}
\alpha rI & (1-\alpha)R\\
(1-\alpha)R^T & 2\alpha r I+(1-\alpha)B
\end{bmatrix},\\
\intertext{where $B$ is the adjacency matrix of line graph of $G$. Then,}
    \phi(A_\alpha(M(G)),\lambda)=&\begin{vmatrix}
    (\lambda-\alpha r)I & -(1-\alpha)R\\
-(1-\alpha)R^T & (\lambda-2\alpha r )I-(1-\alpha)B
    \end{vmatrix}.\\
    \intertext{By Lemma \ref{uvwx}}
    \phi(A_\alpha(M(G)),\lambda)=&(\lambda-\alpha r)^{p-q}\left|(\lambda-2\alpha r)(\lambda-\alpha r)I-(1-\alpha)(\lambda-\alpha r)B-(1-\alpha)^2(B+2I)\right|\\
    =&(\lambda-\alpha r)^{p-q}((\lambda-2\alpha r)(\lambda-\alpha r)+2(1-\alpha)(\lambda-\alpha r))^{q-p}\\
    &\hspace{2cm}\prod_{i=1}^p((\lambda-2\alpha r)(\lambda-\alpha r)-2(1-\alpha)^2-((1-\alpha)(\lambda-\alpha r)+(1-\alpha)^2)(\lambda_i+r-2))\\
    =&\left(\lambda-2\alpha r+2(1-\alpha)\right)^{q-p}\\
    &\hspace{0.5cm}\prod_{i=1}^p\left(\lambda^2-((1-\alpha)(\lambda_i-2)+r(1+2\alpha))\lambda+r(\alpha^2(r-\lambda_i+1)+\alpha(r+\lambda_i)-1)-(1-\alpha)^2\lambda_i\right).
\end{align*}
\end{proof}

Using Proposition \ref{middlep}, we obtain the $A_\alpha$-spectrum of $M(G)$, where $G$ is an $r$ regular graph as follows:

\begin{corollary}
The $A_\alpha$-spectrum of $M(G)$ of an $r$-regular graph is
$$\begin{pmatrix}
    2\alpha r-2(1-\alpha) & x_1 & x_2 \\
    q-p & 1 & 1
\end{pmatrix},$$
where $x_1$ and $x_2$ are $\frac{(1- \alpha)(\lambda_i-2)+r(1+2\alpha )\pm\sqrt{((1-\alpha)\lambda_i+r)^2+4(1-\alpha)(1-\alpha-\alpha r)}}{2}$. 

\end{corollary}

\begin{corollary}
    The adjacency spectrum of $M(G)$ of an $r$-regular graph is $$\begin{pmatrix}
   -2 & \frac{r-2+\lambda_i\pm\sqrt{(r+\lambda_i)^2+4}}{2} \\
    q-p & 1
\end{pmatrix}.$$ 
\end{corollary}

We now present the $A_\alpha$-energy of $M(G)$ in the following corollary.
\begin{corollary}
For $\alpha\in[0,1)$, the $A_\alpha$-energy of $M(G)$ of an $r$-regular graph is
$$\varepsilon_\alpha(M(G))=p(1-\alpha)(r-2)+\sum_{i=1}^p\left|x_1-2\alpha r\right|+\sum_{i=1}^p\left|x_2-2\alpha r\right|,$$
where $x_1$ and $x_2$ are $\frac{(1- \alpha)(\lambda_i-2)+r(1+2\alpha )\pm\sqrt{((1-\alpha)\lambda_i+r)^2+4(1-\alpha)(1-\alpha-\alpha r)}}{2}$.
\end{corollary}

\begin{corollary}
    The adjacency energy of $M(G)$ of an $r$-regular graph is $\displaystyle\varepsilon(M(G))=p(r-2)+\sum_{i=1}^p\sqrt{(r+\lambda_i)^2+4}.$
\end{corollary}

\subsection{Central Graph}
\begin{proposition}\label{centralp}
Let $G$ be an $r$-regular graph on $p$ vertices and $q$ edges. Then the $A_\alpha$-characteristic polynomial of the central graph of $G$ is
\begin{align*}\phi(A_\alpha(C(G)),\lambda)=&(\lambda-2\alpha)^{\frac{p(r-2)}{2}}\Big(\lambda^2+((1-\alpha)(r-p)-(2+p)\alpha+1)\lambda-2(r(1-\alpha)-\alpha(p-1))\Big)\\ 
&\hspace{-1cm}\prod_{i=2}^p\Big(\lambda^2+((1-\alpha)\lambda_i-\alpha(2+p)+1)\lambda-(1-\alpha^2)\lambda_i+(2n-r)\alpha^2-2\alpha(1-r)-r\Big).
\end{align*}
\end{proposition}

\begin{proof}
\begin{align*}
\intertext{The $A_\alpha$ matrix of the central graph of an $r$-regular graph is of the form}
A_\alpha(C(G))&=\begin{bmatrix}
(p\alpha-1)I+(1-\alpha)(J-A) & (1-\alpha)R\\
(1-\alpha)R^T & 2\alpha I
\end{bmatrix}.\\
\intertext{Then}
    \phi(A_\alpha(C(G)),\lambda)=&\begin{vmatrix}
    (\lambda-p\alpha+1)I-(1-\alpha)(J-A) & -(1-\alpha)R\\
    -(1-\alpha)R^T & (\lambda-2\alpha)I
    \end{vmatrix}.\\
    \intertext{By Lemma \ref{uvwx}}
    \phi(A_\alpha(C(G)),\lambda)=&(\lambda-2\alpha)^{q-p}|(\lambda-2\alpha)(\lambda-p\alpha+1)I-(\lambda-2\alpha)(1-\alpha)(J-A)-(1-\alpha)^2(A+rI)|\\
    =&(\lambda-2\alpha)^{q-p}|(\lambda^2-(2\alpha+p\alpha-1)\lambda+2\alpha^2 p-2\alpha-r+2\alpha r-\alpha^2r)I\\
    &\hspace{3cm}-(\lambda(1-\alpha)-2\alpha(1-\alpha))J+(\lambda-\alpha\lambda+2\alpha^2-1-\alpha^2)A|\\
    =&(\lambda-2\alpha)^{q-p}\prod_{i=1}^p\Big((\lambda-(2\alpha+\alpha p-1)\lambda+(2n-r)\alpha^2-2\alpha(1-r)-r)\\
    &\hspace{3cm}-(\lambda-2\alpha)(1-\alpha)P(\lambda_i)+(\lambda(1-\alpha)+\alpha^2-1)\lambda_i\Big)\\
    =&(\lambda-2\alpha)^{\frac{p(r-2)}{2}}\Big(\lambda^2+((1-\alpha)(r-p)-(2+p)\alpha+1)\lambda-2(r+\alpha-\alpha p-r\alpha)\Big)\\ 
    &\hspace{-1cm}\prod_{i=2}^p\Big(\lambda^2+((1-\alpha)\lambda_i-2\alpha-p\alpha+1)\lambda-(1-\alpha^2)\lambda_i+(2p-r)\alpha^2-2\alpha(1-r)-r\Big).
\end{align*}
\end{proof}

Using Proposition \ref{centralp}, we obtain the $A_\alpha$-spectrum of $C(G)$, where $G$ is an $r$ regular graph as follows:
\begin{corollary}
The $A_\alpha$-spectrum of $C(G)$ of an $r$-regular graph consists of:
\begin{enumerate}
    \item $2\alpha$ repeated $\displaystyle\frac{p(r-2)}{2}$ times,
    \item $\alpha+\displaystyle\frac{p-r(1-\alpha)-1}{2}\pm\frac{\sqrt{\alpha^2(r+2)^2+2\alpha(p(r-2)-r(r+7)+2)+(p-r-1)^2+8r}}{2}$ and
    \item $\alpha+\displaystyle\frac{\alpha p-\lambda_i(1-\alpha)-1}{2}$ $\pm\frac{\sqrt{\Big((1-\alpha)\lambda_i-\alpha p+1\Big)^2+4\Big((1-\alpha)\lambda_i+\alpha^2(1+r-4p)+\alpha(1-r)+r\Big)}}{2}.$ 
\end{enumerate}
\end{corollary}

We now present the $A_\alpha$-energy of $C(G)$ in the following corollary.
\begin{corollary}
For $\alpha\in[0,1)$, the $A_\alpha$-energy of $C(G)$ of an $r$-regular graph is
$$\varepsilon_\alpha(C(G))=\frac{p(r-2)\alpha}{r+2}|p-3|+\left|x_1-\frac{2\alpha(p-1 r)}{r+2}\right|+\left|x_2-\frac{2\alpha(p-1 r)}{r+2}\right|+\sum_{i=2}^p\left|y_1-\frac{2\alpha(p-1 r)}{r+2}\right|+\sum_{i=2}^p\left|y_2-\frac{2\alpha(p-1 r)}{r+2}\right|,$$ where $x_i$'s, $i=1,2$, are roots of the equation $$\left(\lambda^2+((1-\alpha)(r-p)-(2+p)\alpha+1)\lambda-2(r(1-\alpha)-\alpha(p-1))\right)=0$$ and $y_j$'s $j=1,2$, are roots of the equation $$\left(\lambda^2+((1-\alpha)\lambda_i-2\alpha-p\alpha+1)\lambda-(1-\alpha^2)\lambda_i+(2p-r)\alpha^2-2\alpha(1-r)-r\right)=0.$$
\end{corollary}

\begin{corollary}
    The adjacency energy of $C(G)$ of an $r$-regular graph is $\displaystyle\varepsilon(C(G))=\sqrt{(n-1-r)^2+8r}+\sum_{i=2}^p\sqrt{(1+\lambda_i)^2+4(r+\lambda_i)}$.
\end{corollary}

\begin{example}
The $A_\alpha$-spectrum of central
graph of $K_p$ is
\begin{enumerate}
    \item $2\alpha$ repeated $\displaystyle\frac{p(p-3)}{2}$,
    \item $\alpha\pm\displaystyle\frac{\sqrt{\alpha^2(p+1)^2+8(p-1)(1-2\alpha)}}{2}$ and
    \item $\displaystyle\frac{\alpha(p-1)}{2}\pm\frac{\sqrt{\alpha^2(p-1)^2+4(3\alpha^2p+\alpha(3-p)+p-2)}}{2}$ repeated $p-1$ times.
\end{enumerate}
\end{example}

\subsection{\texorpdfstring{$m$}{}-Splitting Graph}

\begin{proposition}\label{split}
    Let $G$ be an $r$-regular graph with $p$ vertices and $q$ edges.Then the $A_\alpha$-characteristic polynomial of $m$-splitting graph of $G$ is
    \begin{align*}
    \phi(A_\alpha(Spl_m(G)),\lambda)=&\left(\lambda-\alpha r\right)^{p(m-1)}\\
    &\prod_{i=1}^p\left((\lambda-\alpha r)(\lambda-\alpha(m+1)r)-(1-\alpha)(\lambda-\alpha r)\lambda_i-m(1-\alpha)^2\lambda_i^2\right).
    \end{align*}
\end{proposition}

\begin{proof}
\begin{align*}
\intertext{The $A_\alpha$ matrix of the $m$-splitting of an $r$-regular graph is of the form}
A_\alpha(Spl_m(G))&=\begin{bmatrix}
\alpha(m+1)rI+(1-\alpha)A & (1-\alpha)J_{1\times m}\otimes A\\
(1-\alpha)J_{m\times 1}\otimes A & \alpha r I_{mp}
\end{bmatrix}.\\
\intertext{Then}
    \phi(A_\alpha(Spl_m(G)),\lambda)=&\begin{vmatrix}
    (\lambda-\alpha(m+1)r)I-(1-\alpha)A & -(1-\alpha)J\otimes A\\
-(1-\alpha)J\otimes A & (\lambda-\alpha r) I_{mp}
    \end{vmatrix}.\\
    \intertext{By Lemma \ref{uvwx}}
    \phi(A_\alpha(Spl_m(G)),\lambda)=&(\lambda-\alpha r)^{p(m-1)}|(\lambda-\alpha r)((\lambda-\alpha(m+1)r)I-(1-\alpha)A)-(1-\alpha)^2(J\otimes A)(J\otimes A)|\\
    =&(\lambda-\alpha r)^{p(m-1)}|(\lambda-\alpha r)((\lambda-\alpha(m+1)r)I-(1-\alpha)A)-m(1-\alpha)^2A^2|\\
    =&\left(\lambda-\alpha r\right)^{p(m-1)}\prod_{i=1}^p\left((\lambda-\alpha r)(\lambda-\alpha(m+1)r)-(1-\alpha)(\lambda-\alpha r)\lambda_i-m(1-\alpha)^2\lambda_i^2\right).
\end{align*}
\end{proof}

Using Proposition \ref{split}, we obtain the $A_\alpha$-spectrum of $Spl_m(G)$, where $G$ is an $r$ regular graph as follows:
\begin{corollary}
The $A_\alpha$-spectrum of $Spl_m(G)$ of an $r$-regular graph is
$$\begin{pmatrix}
    \alpha r & x_1 & x_2 \\
    p(m-1) & 1 & 1
\end{pmatrix},$$
where $x_1,x_2=\frac{\alpha r(m+2)+(1-\alpha)\lambda_i\pm\sqrt{(\alpha r(m+2))^2+(1+4m)(1-\alpha)^2\lambda_i^2+2\alpha mr(1-\alpha)\lambda_i}}{2}$.
\end{corollary}

We now present the $A_\alpha$-energy of $Spl_m(G)$ in the following corollary.
\begin{corollary}
For $\alpha\in[0,1)$, the $A_\alpha$-energy of $Spl_m(G)$ of an $r$-regular graph is
$$\varepsilon_\alpha(Spl_m(G))=\sum_{i=1}^p\sqrt{(\alpha r(m+2))^2+(1+4m)(1-\alpha)^2\lambda_i^2+2\alpha mr(1-\alpha)\lambda_i}.$$
\end{corollary}

\subsection{Closed Splitting Graph}
\begin{proposition}\label{closedsplit}
   Let $G$ be an $r$-regular graph on $p$ vertices. Then

\begin{align*}
\phi(A_\alpha  \left(\Lambda(G)\right),\lambda) =&\prod_{i=1}^{p}\left((\lambda-\alpha(1+r)((\lambda-\alpha(1+2r)-(1-\alpha)\lambda_i)-(1-\alpha)^2(\lambda_i+1)^2\right).
\end{align*} 
\end{proposition}

\begin{proof}
    \begin{align*}
\intertext{The $A_\alpha$ matrix of the closed splitting graph of a regular graph is of the form}
A_\alpha(\Lambda(G))&=\begin{bmatrix}
\alpha (2r+1)I+(1-\alpha)A & (1-\alpha)(A+I)\\
(1-\alpha)(A+I) & \alpha (r+1)I
\end{bmatrix}.\\
\intertext{Then,}
\phi(A_\alpha(\Lambda(G)),\lambda)=&\begin{vmatrix}
    \lambda-\alpha (2r+1)I-(1-\alpha)A & -(1-\alpha)(A+I)\\
-(1-\alpha)(A+I) & \lambda-\alpha (r+1)I
    \end{vmatrix}.\\
    \intertext{By Lemma \ref{uvwx}}
    \phi(A_\alpha(\Lambda(G)),\lambda)=&\left|(\lambda-\alpha (2r+1)I-(1-\alpha)A )(\lambda-\alpha (r+1)I)-(1-\alpha)^2(A+I)^2\right|\\
    =&\prod_{i=1}^p\left((\lambda-\alpha(2r+1)-(1-\alpha)\lambda_i)(\lambda-\alpha(r+1))-(1-\alpha)^2(\lambda_i+1)^2\right).
\end{align*}
\end{proof}

Using Proposition \ref{closedsplit}, we obtain the $A_\alpha$-spectrum of $\Lambda(G)$, where $G$ is an $r$ regular graph as follows:

\begin{corollary}
    Let $G$ be an $r$-regular graph with $p$ vertices. Then the $A_\alpha$-spectrum of $\Lambda(G)$ consists of:
    \begin{enumerate}
        \item {\small $\frac{1}{2}\left(2\alpha(1+r)+\lambda_{\alpha_i}+\sqrt{(2\alpha(1+r)+\lambda_{\alpha_i})^2-4\alpha(1+r)(\alpha(1+r)+\lambda_{\alpha_i})+4(1-\alpha)^2(\lambda_i+1)^2}\right)$} for each $i=1,2,\dots,p$,
       \item {\small $\frac{1}{2}\left(2\alpha(1+r)+\lambda_{\alpha_i}-\sqrt{(2\alpha(1+r)+\lambda_{\alpha_i})^2-4\alpha(1+r)(\alpha(1+r)+\lambda_{\alpha_i})+4(1-\alpha)^2(\lambda_i+1)^2}\right)$} for each $i=1,2,\dots,p$.
    \end{enumerate}
\end{corollary}

\begin{corollary}
    Let $G$ be any graph on $p$ vertices. Then the adjacency spectrum of $\Lambda(G)$ consists of:
    \begin{enumerate}
        \item  $\frac{\lambda_{i}+\sqrt{5\lambda_{i}^2+8\lambda_i+4}}{2}$ for each $i=1,2,\dots,p$,
       \item $\frac{\lambda_{i}-\sqrt{5\lambda_{i}^2+8\lambda_i+4}}{2}$ for each $i=1,2,\dots,p$.
    \end{enumerate}
\end{corollary}

We now present the $A_\alpha$-energy of $\Lambda(G)$ in the following corollary.

\begin{corollary}
     For $\alpha\in[0,1)$, the $A_\alpha$-energy of $\Lambda(G)$ of an $r$-regular graph is\\
$\displaystyle \varepsilon_\alpha(\Lambda(G))=\frac{1}{2}\sum_{i=1}^p\left|\left(\lambda_{\alpha_i}-2\alpha\pm\sqrt{(2\alpha(1+r)+\lambda_{\alpha_i})^2-4\alpha(1+r)(\alpha(1+r)+\lambda_{\alpha_i})+4(1-\alpha)^2(\lambda_i+1)^2}\right)\right|.$
\end{corollary}

\begin{corollary}
     The adjacency energy of $\Lambda(G)$ of any graph is $\displaystyle\varepsilon(\Lambda(G))=\sum_{i=1}^p\sqrt{5\lambda_{i}^2+8\lambda_i+4}.$ 
\end{corollary}

\subsection{Closed Shadow Graph}

\begin{proposition}\label{closedshadowp}
   Let $G$ be an $r$-regular graph on $p$ vertices. Then

\begin{align*}
\phi(A_\alpha  \left(D_2[G]\right),\lambda) =&(\lambda-2\alpha(r+1)+1)^p\prod_{i=1}^{p}\left(\lambda-2(1-\alpha)\lambda_i-2\alpha r-1\right).
\end{align*} 
\end{proposition}

\begin{proof}
    \begin{align*}
\intertext{The $A_\alpha$ matrix of the closed shadow graph of a regular graph is of the form}
A_\alpha(D_2[G])&=\begin{bmatrix}
\alpha (2r+1)I+(1-\alpha)A & (1-\alpha)(A+I)\\
(1-\alpha)(A+I) & \alpha (2r+1)I+(1-\alpha)A
\end{bmatrix}.\\
\intertext{Then,}
\phi(A_\alpha(D_2[G]),\lambda)=&\begin{vmatrix}
    (\lambda-\alpha (2r+1))I-(1-\alpha)A & -(1-\alpha)(A+I)\\
-(1-\alpha)(A+I) & (\lambda-\alpha (2r+1))I-(1-\alpha)A
    \end{vmatrix}.\\
    \intertext{By Lemma \ref{uvwx}}
    \phi(A_\alpha(D_2[G]),\lambda)=&\left|(\lambda-\alpha (2r+1)I-(1-\alpha)A )^2-(1-\alpha)^2(A+I)^2\right|\\
    &\left|(\lambda-\alpha (2r+1)I-(1-\alpha)A +(1-\alpha)(A+I))(\lambda-\alpha (2r+1)I-(1-\alpha)A -(1-\alpha)(A+I))\right|\\
    =&\prod_{i=1}^p\left(\lambda-2\alpha(r+1)+1\right)\left(\lambda-2(1-\alpha)\lambda_i-2\alpha r-1\right)\\
    =&(\lambda-2\alpha(r+1)+1)^p\prod_{i=1}^{p}\left(\lambda-2(1-\alpha)\lambda_i-2\alpha r-1\right).
\end{align*}
\end{proof}

Using Proposition \ref{closedshadowp}, we obtain the $A_\alpha$-spectrum of $D_2[G]$, where $G$ is an $r$ regular graph as follows:

\begin{corollary}
    Let $G$ be an $r$-regular graph with $p$ vertices. Then the $A_\alpha$-spectrum of $D_2[G]$ consists of:
    \begin{enumerate}
        \item $2\alpha(r+1)-1$ repeated $p$ times,
       \item $2(1-\alpha)\lambda_i+2\alpha r+1$ for each $i=1,2,\dots,p$.
    \end{enumerate}
\end{corollary}

\begin{corollary}
    Let $G$ be any graph on $p$ vertices. Then the adjacency spectrum of $D_2[G]$ is $$\left(\begin{array}{cc}
        -1 & 2\lambda_i+1 \\
        p & 1
    \end{array}\right), i=1,2,\dots,p.$$
\end{corollary}

We present the $A_\alpha$-energy of $D_2[G]$ in the upcoming corollary.

\begin{corollary}
    For $\alpha\in[0,1)$, the $A_\alpha$-energy of $D_2[G]$ of an $r$-regular graph is
$$\varepsilon_\alpha(D_2[G])=(1-\alpha)\left(p+\sum_{i=1}^p\left|2\lambda_i+1\right|\right).$$
\end{corollary}

\begin{corollary}
    The adjacency energy of $D_2[G]$ of any graph is $\displaystyle\varepsilon(D_2[G])=p+\sum_{i=1}^p\left|2\lambda_i+1\right|$.
\end{corollary}

\subsection{Extended Bipartite Double Graph}

\begin{proposition}\label{extbidoub}
   Let $G$ be an $r$-regular graph on $p$ vertices. Then

\begin{align*}
\phi(A_\alpha  \left(Ebd(G)\right),\lambda) =&\prod_{i=1}^{p}\left(\lambda^2-2\alpha(r+1)\lambda+\alpha^2(r+1)^2-(1-\alpha)^2(\lambda_i+1)^2\right).
\end{align*} 
\end{proposition}

\begin{proof}
    \begin{align*}
\intertext{The $A_\alpha$ matrix of the extended bipartite double graph of a regular graph is of the form}
A_\alpha(Ebd(G))&=\begin{bmatrix}
\alpha (r+1)I& (1-\alpha)(A+I)\\
(1-\alpha)(A+I) & \alpha (r+1)I
\end{bmatrix}.\\
\intertext{Then,}
\phi(A_\alpha(Ebd(G)),\lambda)=&\begin{vmatrix}
    (\lambda-\alpha (r+1))I& -(1-\alpha)(A+I)\\
-(1-\alpha)(A+I) & (\lambda-\alpha (r+1))I
    \end{vmatrix}.\\
    \intertext{By Lemma \ref{uvwx}}
    \phi(A_\alpha(Ebd(G)),\lambda)=&\left|(\lambda-\alpha (r+1))^2I-(1-\alpha)^2(A+I)^2\right|\\
    =&\prod_{i=1}^{p}\left(\lambda^2-2\alpha(r+1)\lambda+\alpha^2(r+1)^2-(1-\alpha)^2(\lambda_i+1)^2\right).
\end{align*}
\end{proof}

Using Proposition \ref{extbidoub}, we obtain the $A_\alpha$-spectrum of $Ebd(G)$, where $G$ is an $r$ regular graph as follows:

\begin{corollary}
    Let $G$ be an $r$-regular graph with $p$ vertices. Then the $A_\alpha$-spectrum of $Ebd(G)$ is
   $$\left(\begin{array}{cc}
      \alpha(r+1)+(1-\alpha)(\lambda_i+1)  & \alpha(r+1)-(1-\alpha)(\lambda_i+1) \\
       1 & 1
   \end{array}\right), i=1,2,\dots,p.$$
\end{corollary}

\begin{corollary}
    Let $G$ be any graph on $p$ vertices. Then the adjacency spectrum of $Ebd(G)$ is $$\left(\begin{array}{cc}
        \lambda_i+1 & -\lambda_i-1 \\
        1 & 1
    \end{array}\right), i=1,2,\dots,p.$$
\end{corollary}

In the following corollary, we introduce the $A_\alpha$-energy of $Ebd(G)$.

\begin{corollary}
    For $\alpha\in[0,1)$, the $A_\alpha$-energy of $Ebd(G)$ of an $r$-regular graph is
$$\varepsilon_\alpha(Ebd(G))=2(1-\alpha)\sum_{i=1}^p\left|\lambda_i+1\right|.$$
\end{corollary}

\begin{corollary}
    The adjacency energy of $Ebd(G)$ of any graph is $\displaystyle\varepsilon(Ebd(G))=2\sum_{i=1}^p\left|\lambda_i+1\right|.$
\end{corollary}

In the following remark, we present the $A_\alpha$-energy of some regular graphs formed from some unary operations on regular graphs.

\begin{remark}
    Since the $A_\alpha$-eigenvalues of an $r$-regular graphs are of the form $\alpha r+(1-\alpha)\lambda_i(A(G))$, their $A_\alpha$-energy can be calculated directly from the equation $\varepsilon_\alpha(G)=(1-\alpha)\varepsilon(G)$.
    \begin{itemize}
        \item The $A_\alpha$-energy of $m$-shadow graph of an $r$-regular graph $G$ is $\varepsilon_\alpha(D_m(G))=m(1-\alpha)\varepsilon(G)$.
        \item The $A_\alpha$-energy of $(k+1)^{th}$ iterated line graph of an $r$-regular graph $G$ is $\displaystyle\varepsilon_\alpha(L^{k+1}(G)) = 2p(r - 2)\prod_{i=1}^{k-1}(2^ir - 2^{i+1} + 2).$
        \item The $A_\alpha$-energy of $m$-duplicate graph of an $r$-regular graph $G$ is $\varepsilon_\alpha(D^m(G))=(1-\alpha)2^m\varepsilon(G)$.
    \end{itemize}
\end{remark}

\section{Observations}
The graphs $Spl(G),\Lambda(G),D_2[G],Ebd(G),D_2(G),D(G)$ has same number of vertices, that is $2p$ vertices, where $p$ is the order of $G$. In Table \ref{unaryenergy}, with the help of Matlab software, we find the $A_\alpha$-energy of some graphs as $\alpha$ varies. 

In literature, there are only a few graphs found that are $A_\alpha$-equienergetic or $A_\alpha$-bordereneregtic. From Table \ref{unaryenergy} we identify some graphs of this kind.

\begin{itemize}
    \item For all values of $\alpha$, $D_2(G)$ and $D(G)$ are $A_\alpha$-equienergetic. 
    \item  $D_2[C_4]$ is $A_\alpha$-borderenergetic for all values of $\alpha$.
    \item $D_2[C_6]$ and $D_2[K_{3,3}]$ are both $A_\alpha$-equienergetic and $A_\alpha$-borderenergetic.
    \item $Ebd(C_6)$ is $A_\alpha$-equienergetic with $D_2(C_6)$ and $D(C_6)$.
    \item $D_2[K_{p,p}]$ is $A_\alpha$-borderenergetic for all values of $p$ and $\alpha$.
    \item For $\alpha\geq 0.3$ $Spl(G)$ is hyperenergetic.
    
\end{itemize}

\begin{table}[H]
      \centering
    \begin{tabular}{ccccccccccc}
    \hline
    &0&0.1&0.2&0.3&0.4&0.5&0.6&0.7&0.8&0.9\\
    \hline
    $K_8$	&	14	&	12.6	&	11.2	&	9.8	&	8.4	&	7	&	5.6	&	4.2	&	2.8	&	1.4	\\
         $Spl(C_4)$	&	8.9443	&	9.3369	&	9.9395	&	10.8071	&	11.9801	&	13.4641	&	15.2257	&	17.2099	&	19.3617	&	21.6362	\\
$\Lambda(C_4)$	&	13.153	&	11.7889	&	10.4985	&	9.3106	&	8.2676	&	7.434	&	6.903	&	6.737	&	6.8958	&	7.3227	\\
$D_2[C_4]$	&	14	&	12.6	&	11.2	&	9.8	&	8.4	&	7	&	5.6	&	4.2	&	2.8	&	1.4	\\
$Ebd(C_4)$	&	12	&	10.8	&	9.6	&	8.4	&	7.2	&	6	&	4.8	&	3.6	&	2.4	&	1.2	\\
$D_2(C_4)$	&	8	&	7.2	&	6.4	&	5.6	&	4.8	&	4	&	3.2	&	2.4	&	1.6	&	0.8	\\
$D(C_4)$	&	8	&	7.2	&	6.4	&	5.6	&	4.8	&	4	&	3.2	&	2.4	&	1.6	&	0.8	\\
\hline
$K_{10}$	&	18	&	16.2	&	14.4	&	12.6	&	10.8	&	9	&	7.2	&	5.4	&	3.6	&	1.8	\\
$Spl(C_5)$	&	14.4721	&	13.5192	&	13.3638	&	13.9305	&	15.1374	&	16.8829	&	19.0463	&	21.5154	&	24.2026	&	27.0452	\\
$\Lambda(C_5)$	&	16.986	&	15.1326	&	13.3447	&	11.8961	&	10.5946	&	9.3861	&	8.4907	&	8.3826	&	8.6148	&	9.1532	\\
$D_2[C_5]$	&	18.9443	&	17.0498	&	15.1554	&	13.261	&	11.3666	&	9.4721	&	7.5777	&	5.6833	&	3.7889	&	1.8944	\\
$Ebd(C_5)$	&	14.9443	&	13.4498	&	11.9554	&	10.461	&	8.9666	&	7.4721	&	5.9777	&	4.4833	&	2.9889	&	1.4944	\\
$D_2(C_5)$	&	12.9442	&	11.6498	&	10.3554	&	9.0609	&	7.7665	&	6.4721	&	5.1777	&	3.8833	&	2.5888	&	1.2944	\\
$D(C_5)$	&	12.9442	&	11.6498	&	10.3554	&	9.0609	&	7.7665	&	6.4721	&	5.1777	&	3.8833	&	2.5888	&	1.2944	\\
\hline	
$K_{12}$	&	22	&	19.8	&	17.6	&	15.4	&	13.2	&	11	&	8.8	&	6.6	&	4.4	&	2.2	\\
$Spl(C_6)$	&	17.8885	&	16.5299	&	16.1352	&	16.7224	&	18.156	&	20.2551	&	22.8544	&	25.8183	&	29.043	&	32.4543	\\
$\Lambda(C_6)$	&	19.3992	&	17.4954	&	15.6352	&	13.8385	&	12.1401	&	10.6056	&	10.144	&	10.0525	&	10.3368	&	10.9838	\\
$D_2[C_6]$	&	22	&	19.8	&	17.6	&	15.4	&	13.2	&	11	&	8.8	&	6.6	&	4.4	&	2.2	\\
$Ebd(C_6)$	&	16	&	14.4	&	12.8	&	11.2	&	9.6	&	8	&	6.4	&	4.8	&	3.2	&	1.6	\\
$D_2(C_6)$	&	16	&	14.4	&	12.8	&	11.2	&	9.6	&	8	&	6.4	&	4.8	&	3.2	&	1.6	\\
$D(C_6)$	&	16	&	14.4	&	12.8	&	11.2	&	9.6	&	8	&	6.4	&	4.8	&	3.2	&	1.6	\\
$Spl(K_{3,3})$	&	13.4164	&	15.8053	&	18.5093	&	21.6107	&	25.1702	&	29.1962	&	33.6385	&	38.4149	&	43.4426	&	48.6542	\\
$\Lambda(K_{3,3})$	&	21.544	&	19.426	&	17.575	&	16.0566	&	14.9225	&	14.2111	&	13.9742	&	14.2751	&	15.1022	&	16.3675	\\
$D_2[K_{3,3}]$	&	22	&	19.8	&	17.6	&	15.4	&	13.2	&	11	&	8.8	&	6.6	&	4.4	&	2.2	\\
$Ebd(K_{3,3})$	&	20	&	18	&	16	&	14	&	12	&	10	&	8	&	6	&	4	&	2	\\
$D_2(K_{3,3})$	&	12	&	10.8	&	9.6	&	8.4	&	7.2	&	6	&	4.8	&	3.6	&	2.4	&	1.2	\\
$D(K_{3,3})$	&	12	&	10.8	&	9.6	&	8.4	&	7.2	&	6	&	4.8	&	3.6	&	2.4	&	1.2	\\
\hline
% $K_{14}$	&	26	&	23.4	&	20.8	&	18.2	&	15.6	&	13	&	10.4	&	7.8	&	5.2	&	2.6	\\
% $Spl(C_7)$	&	20.0976	&	18.6631	&	18.5215	&	19.4121	&	21.1596	&	23.6271	&	26.663	&	30.1213	&	33.8836	&	37.8633	\\
% $\Lambda(C_7)$	&	22.8684	&	20.3588	&	17.9978	&	15.8625	&	14.205	&	12.8622	&	11.9008	&	11.7361	&	12.0602	&	12.8145	\\
% $D_2[C_7]$	&	24.4155	&	21.974	&	19.5324	&	17.0909	&	14.6493	&	12.2078	&	9.7662	&	7.3247	&	4.8831	&	2.4416	\\
% $Ebd(C_7)$	&	20.4155	&	18.374	&	16.3324	&	14.2909	&	12.2493	&	10.2078	&	8.1662	&	6.1247	&	4.0831	&	2.0416	\\
% $D_2(C_7)$	&	17.9758	&	16.1782	&	14.3806	&	12.5831	&	10.7855	&	8.9879	&	7.1903	&	5.3927	&	3.5952	&	1.7976	\\
% $D(C_7)$	&	17.9758	&	16.1782	&	14.3806	&	12.5831	&	10.7855	&	8.9879	&	7.1903	&	5.3927	&	3.5952	&	1.7976	\\
% \hline
    \end{tabular}
    \caption{$A_\alpha$-energy of some graphs for different values of $\alpha$.}
    \label{unaryenergy}
\end{table}

\section*{Conclusion}
In this paper, we derive the $A_\alpha$-characteristic polynomial of some unary operations on graphs such as the middle graph, the central graph, the m-splitting, the closed splitting graph, the m-shadow, the closed shadow, the extended bipartite double graph, the iterated line graph and the m-duplicate graph. Using these results we computed their $A_\alpha$-energy. Furthermore from our observations, we found graphs that are $A_\alpha$-equienergetic and $A_\alpha$-borderenergetic.

\bibliographystyle{unsrt}  
\bibliography{references}  

\end{document}